\documentclass[12pt,reqno]{amsart}
\usepackage{geometry}                
\usepackage{amsmath,amssymb}
\usepackage{graphicx}
\usepackage{amssymb,latexsym, amsmath}
\usepackage{amsfonts,amsthm}
\usepackage{amscd}
\usepackage{mathrsfs}
\usepackage{hyperref}
\usepackage{eucal}
\usepackage{epstopdf}
\usepackage{amsgen}
\usepackage{xspace}
\usepackage{verbatim}
\usepackage{stmaryrd}
\usepackage{enumitem}
\newlist{steps}{enumerate}{1}
\setlist[steps, 1]{label = Step \arabic*:}

\newcommand{\gd}{\Delta}

\newcommand{\inpt}[1]{\langle #1 \rangle}

\newcommand{\gw}{\Omega}

\newcommand{\ga}{\gamma}

\newcommand{\G}{\Gamma}
\newcommand{\gl}{\lambda}

\newcommand{\gs}{\sigma}

\newcommand{\ms}{\mathscr}

\newcommand{\nb}{\nabla}
\newcommand{\vp}{\varphi}
\newcommand{\ve}{\varepsilon}
\newcommand{\pdr}{\partial}

\newcommand{\beq}{\begin{equation}}
\newcommand{\eeq}{\end{equation}}
\newcommand{\bea}{\begin{align}}
\newcommand{\eea}{\end{align}}
\newcommand{\bthm}{\begin{theorem}}
\newcommand{\ethm}{\end{theorem}}
\newcommand{\bpr}{\begin{proof}}
\newcommand{\epr}{\end{proof}}
\newcommand{\bcl}{\begin{corollary}}
\newcommand{\ecl}{\end{corollary}}
\newcommand{\bpn}{\begin{proposition}}
\newcommand{\epn}{\end{proposition}}
\newcommand{\bre}{\begin{remark}}
\newcommand{\ere}{\end{remark}}
\newcommand{\bdf}{\begin{definition}}
\newcommand{\edf}{\end{definition}}
\newcommand{\bss}{\begin{align*}}
\newcommand{\ess}{\end{align*}}

\newcommand{\bl}{\label}

\newcommand{\mR}{\mathbb{R}}
\newcommand{\mH}{\mathbb{H}}

\newcommand{\mE}{\mathbb{E}}

\newtheorem{theorem}{Theorem}[section]
\newtheorem{corollary}[theorem]{Corollary}

\newtheorem{proposition}[theorem]{Proposition}

\theoremstyle{definition}
\newtheorem{definition}[theorem]{Definition}
\theoremstyle{remark}
\newtheorem{remark}{Remark}

\numberwithin{equation}{section}

\begin{document}

\title[FitzHugh-Nagumo Neural Networks]{Dynamics and Synchronization of Boundary Coupled FitzHugh-Nagumo Neural Networks}

\author[L. Skrzypek]{Leslaw Skrzypek}
\address{Department of Mathematics and Statistics, University of South Florida, Tampa, FL 33620, USA}
\email{skrzypek@usf.edu}
\thanks{}

\author[Y. You]{Yuncheng You}
\address{Department of Mathematics and Statistics, University of South Florida, Tampa, FL 33620, USA}
\email{you@mail.usf.edu}
\thanks{}


\date{May 27, 2020}



\begin{abstract} 
In this work a new mathematical model for complex neural networks is presented by the partly diffusive FitzHugh-Nagumo equations with ensemble boundary coupling. We analyze the dissipative dynamics and boundary coupling dynamics of the solution semiflow with sharp estimates. The exponential synchronization of this kind complex neural networks is proved under the condition that synaptic stimulation signal strength reaches a threshold quantitatively expressed. 
\end{abstract}

\maketitle

\section{\textbf{Introduction}}

Synchronization of biological neural networks plays a significant role in the activities of the brain and the central nervous system. Study of synchronization and desynchronization mechanisms with possible control and intervention by means of mathematical models and analysis is one of the central topics in neuroscience and medical science, also in the theory of artificial neural networks. Researches demonstrated that fast synchronization may lead to enhanced functionality and performance of central nervous system. Though it may possibly lead to functional disorders of neuron systems and to cause pathological patterns like epilepsy and Parkinson's disease \cite{HBB, JC, L, TBSS}.

In recent years, the dynamical behavior and problems of complex and large-scale networks including convolutional neural networks in machine learning and deep learning, Internet networks, epidemic spreading networks, collaborative or competitive networks, and social networks attract more and more research interests, cf. \cite{A, B, I, PS, W}. 

Approximately 86 billion neurons can be found in human nervous system and they are connected with approximately $10^{14}$ synapses. Neurons are the nerve cells which form the major pathways of brain wave transmission and communication. The complex neural networks are capable to process, coordinate, and integrate tremendous amount of synaptic signals.  

Neuron signals are short electrical pulses called spike or action potential. Synaptic pulse inputs received by neuron dendrites modify the intercellular membrane potentials and may cause bursting in alternating phases of rapid firing and then refractory quiescence. Neuronal signals triggered at the axon hillock can propagate along the axon and diffuse to neighbor neurons across synapses. In a complex neural network, stimulation signals through synaptic couplings in ensemble neuron activities have to reach certain threshold for achieving synchronization \cite{DJ, SK}. The current mathematical models for synchronization of biological neural networks are either ordinary or partial differential equations without exhibition of boundary coupling. 

The four-dimensional Hodgkin-Huxley model \cite{HH1, HH2} (1952) is highly nonlinear and elaborates the activity of the giant squid axon in terms of the cell membrane potential coupled with three equations of the voltage-gated K$^+$ current, transient Na$^+$ current, and the leak current of other ions. 

The simplified two-dimensional FitzHugh-Nagumo model \cite{FH, NM} (1961-1962) is reduced from the Hodgkin-Huxley equations and mathematically captures the essential neuron properties of excitation and bio-electric transmission with only two equations governing the membrane potential and the ionic current. This famous model yields exquisite and effective phase-plane analysis to demonstrate the refractorily periodic excitation of neurons in response to suprathreshold input pulse, similar to the oscillation-relaxation dynamics exhibited by the Van der Pol equation. However this 2D model cannot generate solutions showing self-sustained chaotic bursting. 

Another model is the three-dimensional Hindmarsh-Rose equations \cite{HR} (1984) and the diffusive or partly diffusive Hindmarsh-Rose equations recently proposed and studied \cite{DH, ET, Phan, PY} on topics of regular and chaotic bursting dynamics, global attractors, and random attractors.

Synchronization for biological neural networks has been studied by using several mathematical models and methods. Most published results are for the FitzHuigh-Nagumo networks of neurons coupled by gap junctions or called space-clamped coupling \cite{AC, IJ, WLZ, Yong} featuring the linear coupling $C \sum_{j (\neq i)} a_{ij} (x_j (t) - x_i (t))$ in the membrane potential equation for the $i$-th neuron. The mean field models of Hodgkin-Huxley and FitzHuigh-Nagumo neuron networks may or may not with noise were studied in \cite{BF, Dick, QT, Tr} replacing the above sum of couplings by its average. Synchronization and control of the diffusive FitzHugh-Nagumo type neural networks have been investigated in \cite{AA, YCY}, which consists of multiple chain-like neurons with the distributed coupling terms $\alpha_i (u_{i-1}(t,x) - u_i (t,x))$ and $\beta_i (v_{i-1}(t,x) - v_i (t,x))$ in the two reaction-diffusion equations for the $i$-th neuron, where $x$ is in the interior of a spatial domain, or by the pointwise pinning-impulse controllers also in the interior of a spatial domain. Synchronization of chaotic neural networks and stochastic neural networks has also been studied in \cite{Dick, SG, WZD},

Beside synchronization of two coupled Hindmarsh-Rose neurons has been studied in \cite{DH, CY}. Recently we have proved in \cite{PyY} the asymptotic synchronization of the star-like Hindmarsh-Rose neural networks.

In this paper,we shall present a mathematical model of FitzHugh-Nagumo neural networks featuring the partly diffusive FitzHugh-Nagumo equations with ensemble boundary coupling based on the Fick's law and the Kirchhoff's law of bio-electrical and biochemical synapses crossing the neuron boundaries. In the rest of the paper, we shall formulate the mathematical framework of this model in Section 2, then prove the dissipative dynamics in Section 3 and the coupling dynamics in Section 4 of the solution semiflows in the basic space $H = L^2 (\gw, \mathbb{R}^2)$ and the regular space $E = H^1(\gw) \times L^2(\gw)$ as well as the boundary trace space. Finally we reach the proof of the main result on asymptotic synchronization of this type complex neural networks in Section 5.

\section{\textbf{FitzHugh-Nagumo Neural Networks with Boundary Coupling}}

\vspace{3pt}
We shall consider a network of boundary coupled neurons denoted by $\mathcal{N}_i, 1 \leq i \leq m,$ modeled by the partly diffusive FitzHugh-Nagumo equations,
\beq \bl{cHR}
\begin{split}
	\frac{\pdr u_i}{\pdr t} & = d \gd u_i + f(u_i, x) - \gs w_i + J,  \quad 1 \leq i \leq m, \\
	\frac{\pdr w_i}{\pdr t} & = \ve (u_i + a - b w_i), \quad 1 \leq i \leq m,
\end{split}
\eeq
for $t > 0,\; x \in \gw \subset \mathbb{R}^{n}$ ($n \leq 3$), the network size $m \geq 2$ is an integer, the spatial domain $\gw$ is bounded and its boundary denoted by
$$
	\partial \gw = \G = \bigcup_{j = 1}^m \G_{ij}, \quad \text{for} \;\, i = 1, \cdots, m,
$$
is locally Lipschitz continuous, where $\G_{ij} = \G_{ji}$ and for each given $i \in \{1, \cdots, m\}$ the boundary pieces $\{\G_{ij}: j = 1, \cdots, m\}$  are measurable and mutually non-overlapping. 

Here $(u_i (t,x), w_i (t,x)), \,i = 1, \cdots, m,$ are the state variables for the $i$-th neuron $\mathcal{N}_i$ in this network, which is coupled with the other neurons $\{\mathcal{N}_j: j \neq i \}$ in the network through the boundary conditions 
\beq \label{nbc}
	\frac{\pdr u_i}{\pdr \nu} (t, x) + p u_i = pu_j, \quad \text{for}  \;\, x \in \G_{ij}, \quad j \in \{1, \cdots, m\},
\eeq
for each of the neurons indexed by $1 \leq i \leq m$ in the network, where $\pdr/\pdr \nu$ stands for the normal outward derivative, $p > 0$ is the coupling strength constant. By \eqref{nbc}, $\frac{\pdr u_i}{\pdr \nu} (t, x)  = 0$ for $x \in \G_{ii}, 1 \leq i \leq m$. 

The initial conditions of the equations \eqref{cHR} to be specified will be denoted by
\beq \bl{inc}
	u_i(0, x) = u_i^0 (x), \;\, w_i (0, x) = w_i^0 (x), \quad x \in \gw, \;\, 1 \leq i \leq m.
\eeq
All the parameters in this system \eqref{cHR} are positive constants except the external current $J \in \mR$. Here we assume that $J > 0$ for notational simplicity, since otherwise it can be replaced by $|J|$ in the sequel argument. 

We make the following Assumption: The scalar function $f \in C^1 (\mathbb{R} \times \gw)$ satisfies the properties:
\beq \bl{Asp}
	\begin{split}
	f(s, x) s &\leq  - \gl |s|^4 + \vp (x), \quad s \in \mathbb{R}, \; x \in \gw, \\[5pt]
	| f(s, x) | &\leq \alpha |s|^3 + \zeta (x), \quad s \in \mathbb{R}, \; x \in \gw, \\
	\left|\frac{\pdr f}{\pdr s} (s, x) \right| &\leq \beta |s|^2 + \xi (x) \quad \text{and} \quad \frac{\pdr f}{\pdr s} (s, x) \leq \ga,  \quad s \in \mathbb{R}, \; x \in \gw, 
	\end{split}
\eeq
where $\gl, \alpha, \beta$ and $\ga$ are positive constants, and $\vp, \zeta, \xi \in L^2 (\gw)$ are given functions. 

In this system \eqref{cHR} of the FitzHugh-Nagumo neural network, for the $i$-th neuron $\mathcal{N}_i$,  $u_i(t,x)$ is the fast \emph{excitatory} variable representing the transmembrane electrical voltage of the neuron cell  and $w_i(t, x)$ is the slow \emph{recovering} variable describing the ionic currents. This mathematical model of the partly diffusive FitzHugh-Nagumo neural networks is a system of partial-ordinary differential equations featuring the Robin-type boundary conditions, which characterizes the neuron dynamics on the interior domain as well as the combination of the Fick's law and the Kirchhoff's law crossing the coupled neuron boundaries. This model also reflects that biophysical signals through synapses are mainly taken place related to the $u_i$-equations for neuron membrane potentials. 

In this study of the initial-boundary value problem \eqref{cHR}-\eqref{inc} of the FitzHugh-Nagumo neural networks, we define the following Hilbert spaces for each of the subsystems representing the involved single neurons
$$
	H = L^2 (\gw, \mR^2), \quad  \text{and} \quad E = H^1 (\gw) \times L^2 (\gw ),
$$
and the corresponding product Hilbert spaces 
$$
	\mathbb{H} =  [L^2 (\gw, \mR^2)]^{m} \quad \text{and} \quad \mathbb{E} = [H^1 (\gw) \times L^2 (\gw)]^{m}
$$
for the entire network system \eqref{cHR}-\eqref{inc}. The norm and the inner-product of the Hilbert spaces $\mathbb{H}, \, H,$ or $L^2(\gw)$ will be denoted by $\| \, \cdot \, \|$ and $\inpt{\,\cdot , \cdot\,}$, respectively. The norm of $\mathbb{E}$ or $E$ will be denoted by $\| \, \cdot \, \|_\mathbb{E}$ or $\| \, \cdot \, \|_E$. We use $| \, \cdot \, |$ to denote the vector norm or the Lebesgue measure of a set in $\mR^n$.

Then the initial-boundary value problem \eqref{cHR}-\eqref{inc} is formulated as an initial value problem of the evolutionary equation:
\begin{equation} \label{pb}
	\begin{split}
	\frac{\partial g_i}{\partial t} = A_i g_i &\, + F(g_i), \quad t > 0, \;\, 1 \leq i \leq m, \\
	& g_i (0) = g_i^0 \in H.
	\end{split}
\end{equation}
Here $g_i(t) = \text{col}\, (u_i(t, \cdot), w_i(t, \cdot))$. The initial data (or called initial states) are functions of the spatial variable $x \in \gw$ and  denoted by $g_i^0 = \text{col}\, (u_i^0, w_i^0)$. The nonpositive, self-adjoint operator $A = diag \, (A_1, \cdots, A_m)$ is defined by the block operator
\begin{equation} \label{opA}
	A_i = 
\begin{bmatrix}
d \gd \quad & 0  \\[3pt]
0 \quad & - \ve b I 
\end{bmatrix}, \quad 1 \leq i \leq m,
\end{equation}
with its domain
$$
	D(A) = \{ \text{col}\, (h_1, \cdots, h_m) \in [H^2(\gw) \times L^2 (\gw)]^{m}: \text{\eqref{nbc} satisfied}\}.
$$ 
By the Assunption \eqref{Asp}, the continuous Sobolev imbedding $H^{1}(\gw) \hookrightarrow L^6(\gw)$ for space dimension $n \leq 3$ and  the H\"{o}lder inequality, it is seen that the nonlinear mapping 
\begin{equation} \label{opf}
	F(g_i) =
\begin{pmatrix}
f(u_i, x) - \gs w_i + J \\[3pt]
\ve (u_i + a) \\[3pt]
\end{pmatrix}
: E \longrightarrow H 
\end{equation}
is locally Lipschitz continuous for $1 \leq i \leq m$. 

Note that the original FitzHugh-Nagumo equations in \cite{FH, NM} were given by the ordinary differential equations
\begin{equation*}
	\begin{split}
	\frac{du}{dt} &= u - \frac{u^3}{3} - w + J, \quad \text{or} \quad \frac{du}{dt} = \kappa u(u- c)(1 - u) - \gs w + J \\
	\frac{dw}{dt} &= \frac{1}{\tau} (u + a - bw).
	\end{split}
\end{equation*}
It is easy to check that the nonlinearity in the above $u$-equation really satisfies the Assumption \eqref{Asp}. 

Here we consider the weak solutions of this initial value problem \eqref{pb}, cf. \cite[Section XV.3]{CV} and the corresponding definition we presented in \cite{CY, PyY}. The following proposition claiming the local existence and uniqueness of weak solutions in time can be proved by the Galerkin approximation method.

\begin{proposition} \label{pps}
	For any given initial state $(g_1^0, \cdots , g_m^0) \in \mathbb{H}$, there exists a unique weak solution $(g_1 (t, g_1^0), \cdots, g_m (t, g_m^0))$ local in time $t \in [0, \tau]$, for some $\tau > 0$, of the initial value problem \eqref{pb} formulated from the initial-boundary value problem \eqref{cHR}-\eqref{inc}. The weak solution continuously depends on the initial data and satisfies 
	\begin{equation} \label{soln}
	(g_1, \cdots, g_m)  \in C([0, \tau]; \,\mH) \cap C^1 ((0, \tau); \,\mH) \cap L^2 ([0, \tau]; \,\mE).
	\end{equation}
If the initial state is in $\mE$, then the solution is a strong solution with the regularity
	\begin{equation} \bl{ss}
	(g_1, \cdots, g_m) \in C([0, \tau]; \,\mE) \cap C^1 ((0, \tau); \,\mE) \cap L^2 ([0, \tau]; \,D(A)).
	\end{equation}
\end{proposition}

The basics of autonomous dynamical systems or called semiflow generated by the evolutionary differential equations are referred to \cite{CV, SY, Tm}. Below is a key concept for dissipative dynamical systems we shall address in this work.

\begin{definition} \label{Dabsb}
	Let $\{S(t)\}_{t \geq 0}$ be a semiflow on a Banach space $\ms{X}$. A bounded set (usually a bounded ball) $B^*$ of $\ms{X}$ is called an absorbing set of this semiflow, if for any given bounded set $B \subset \ms{X}$ there exists a finite time $T_B \geq 0$ depending on $B$, such that $S(t)B \subset B^*$ permanently for all $t \geq T_B$. A semiflow is called dissipative if there exists an absorbing set.
\end{definition}

\section{\textbf{Dissipative Dynamics of the Solution Semiflow}}

In this section, first we prove the global existence in time of weak solutions for the formulated initial value problem \eqref{pb} of the boundary coupled FitzHugh-Nagumo neural networks \eqref{cHR}-\eqref{inc}. Then we show that there exists an absorbing set in the space $H$ for the solution semiflow.

\begin{theorem} \label{Tm}
	For any given initial state $(g_1^0, \cdots , g_m^0) \in \mathbb{H}$, there exists a unique global weak solution $(g_1 (t, g_1^0), \cdots, g_m (t, g_m^0))$ in the space $\mH, \, t \in [0, \infty)$, of the initial value problem \eqref{pb} for the boundary coupled FitzHugh-Nagumo neural network \eqref{cHR}-\eqref{inc}. 
\end{theorem}

\begin{proof}
	Conduct the $L^2$ inner-products of the $u_i$-equation with $C_1 u_i(t)$ for $1 \leq i \leq m$, where the constant $C_1 > 0$ is to be chosen later, and then sum them up to get		
	\begin{equation*}
	\begin{split}
	&\frac{C_1}{2} \frac{d}{dt} \sum_{i = 1}^m \|u_i\|^2 + C_1 d\, \sum_{i=1}^m \|\nb u_i\|^2  = -  d\,C_1p \, \sum_{i=1}^m \sum_{j=1}^m \int_{\G_{ij}} ( u_i - u_j)^2\, dx \\
	&+ C_1 \sum_{i=1}^m \int_\gw (f(u_i, x) u_i - \gs u_i w_i +Ju_i) \, dx  \\
	\leq &\, C_1 \sum_{i=1}^m \int_\gw \left[ - \gl |u_i (t, x)|^4 + |\vp (x)| - \gs u_i (t, x) w_i (t, x) +Ju_i (t, x) \right] dx \\
	\leq &\, C_1 \sum_{i=1}^m \int_\gw \left[ - \gl u_i^4 + |\vp (x) | + \frac{1}{2}\left(\gl u_i^2 + \frac{\gs^2}{\gl} w_i ^2 \right) + \frac{1}{2} \left(\frac{J^2}{\gl} + \gl u_i^2\right) \right] dx \\[2pt]
	= &\, C_1 \sum_{i=1}^m \int_\gw \gl (u_i^2 - u_i^4)\, dx + \frac{C_1 \gs^2}{2\gl}\sum_{i=1}^m \|w_i \|^2  + C_1 m \left(\|\vp \|_{L^1} + \frac{J^2}{2\gl} |\gw| \right) \\[2pt]
	\leq &\, - \frac{1}{2} C_1 \gl \sum_{i=1}^m \int_\gw u_i^4 (t, x)\, dx + \frac{C_1 \gs^2}{2\gl}\sum_{i=1}^m \|w_i \|^2  + C_1 m \left(\|\vp \| |\gw|^{1/2} + \left(\frac{\gl}{2} + \frac{J^2}{2\gl}\right)|\gw| \right)
	\end{split}
	\end{equation*}
where we used the Gauss divergence theorem, the boundary coupling condition \eqref{nbc} and the Assumption \eqref{Asp}. In the last step, $u_i^2 \leq \frac{1}{2} (1 + u_i^4)$ is used.

Then sum up the $L^2$ inner-products of the $w_i$-equations with $w_i(t)$ for $1 \leq i \leq m$, by using Young's inequality, we obtain

	\begin{equation*} 
	\begin{split}
	&\frac{1}{2} \frac{d}{dt} \sum_{i=1}^m \| w_i \|^2 = \sum_{i=1}^m \int_\gw (\ve u_i w_i + \ve a w_i - \ve bw_i^2) \, dx  \\
	\leq &\, \sum_{i=1}^m \int_\gw  \left[\left(\frac{\ve}{b} u_i^2 + \frac{1}{4} \ve b \,w_i^2\right)  + \left(\frac{\ve a^2}{b} + \frac{1}{4} \ve b \,w_i^2\right) - \ve b \,w_i^2\right] dx \\
	= &\, \sum_{i=1}^m \frac{\ve}{b} \int_\gw u_i^2 (t, x)\, dx - \frac{1}{2} \ve b \, \sum_{i=1}^m \|w_i\|^2 + \frac{m \ve a^2}{b} |\gw |.
	\end{split}
	\end{equation*}
Now add the above two inequalities to obtain 
\beq \bl{uw}
	\begin{split}
	& \frac{1}{2}\frac{d}{dt} \sum_{i = 1}^m \left(C_1 \|u_i\|^2 +  \| w_i \|^2 \right) + C_1 d\, \sum_{i=1}^m \|\nb u_i\|^2 +  d\,C_1p \, \sum_{i=1}^m \sum_{j=1}^m \int_{\G_{ij}} ( u_i - u_j)^2\, dx \\
        \leq &\, - \sum_{i=1}^m \int_\gw \left(\frac{1}{2} C_1 \gl u_i^4 - \frac{\ve}{b} u_i^2 \right) dx + \sum_{i=1}^m \left( \frac{C_1 \gs^2}{2\gl} - \frac{1}{2} \ve b\right) \|w_i \|^2 \\
        & + C_1 m \left(\|\vp \| |\gw|^{1/2} + \left(\frac{\gl}{2} + \frac{J^2}{2\gl}\right)|\gw| \right) + \frac{m \ve a^2}{b} |\gw |,  \quad t \in I_{max} = [0, T_{max}),
	\end{split}
\eeq
where $I_{max}$ is the maximal existence interval of a weak solution. Choose constant 
$$
	C_1 = \frac{\ve b \gl}{2\gs^2} \quad \text{so that} \quad \frac{C_1 \gs^2}{2\gl} - \frac{\ve b}{2} = - \frac{\ve b}{4}.
$$
With this choice, from \eqref{uw} it follows that
\beq \bl{u1w}
	\begin{split}
	& \frac{1}{2}\frac{d}{dt} \sum_{i = 1}^m \left(C_1 \|u_i\|^2 +  \| w_i \|^2 \right) + C_1 d\, \sum_{i=1}^m \|\nb u_i\|^2 + \frac{1}{4} \sum_{i=1}^m \ve b \,\|w_i \|^2 \\
	& + \sum_{i=1}^m \int_\gw \left(\frac{1}{2} C_1 \gl u_i^4 - \frac{\ve}{b} u_i^2 \right) dx + d\,C_1p \, \sum_{i=1}^m \sum_{j=1}^m \int_{\G_{ij}} ( u_i - u_j)^2\, dx \\
	\leq &\, C_1 m \|\vp \|^2 + m \left(C_1 + \frac{C_1 \gl}{2} + \frac{C_1 J^2}{2\gl} + \frac{\ve a^2}{b}\right) |\gw |,  \quad  t \in I_{max} = [0, T_{max}).
	\end{split}
\eeq
By the choice of $C_1$ and completing square, we have

\begin{align*}
	\frac{1}{2} C_1 \gl u_i^4 - \frac{2\ve}{b} u_i^2 &= \frac{\ve b \gl^2}{4 \gs^2} u_i^4 - \frac{2\ve}{b} u_i^2 = \ve b \left(\frac{\gl^2}{4 \gs^2}u_i^4 - \frac{2}{b^2} u_i^2 \right) \\
	& = \ve b \left( \frac{\gl}{2\gs} u_i^2 - \frac{2\gs}{\gl b^2}\right)^2 - \frac{4 \ve \gs^2}{\gl^2 b^3},
\end{align*}
so that 
\beq \bl{kiq}
	\begin{split}
	\sum_{i=1}^m \int_\gw \left(\frac{1}{2} C_1 \gl u_i^4 - \frac{\ve}{b} u_i^2 \right) dx &=  \sum_{i=1}^m \frac{\ve}{b} \,\|u_i \|^2 + \sum_{i=1}^m \int_\gw \left(\frac{1}{2} C_1 \gl u_i^4 - \frac{2\ve}{b} u_i^2 \right) dx \\
	&\geq \sum_{i=1}^m \frac{\ve}{b} \,\|u_i \|^2 - \frac{4m \ve \gs^2}{\gl^2 b^3} |\gw |.
	\end{split}
\eeq
Substitute \eqref{kiq} into the integral term over $\gw$ in \eqref{u1w}. It yields the inequality 
\beq \bl{u2w}
	\begin{split}
	& \frac{d}{dt} \sum_{i = 1}^m \left(C_1 \|u_i\|^2 +  \| w_i \|^2 \right) + 2C_1 d\, \sum_{i=1}^m \|\nb u_i\|^2 + \frac{2\ve}{b} \sum_{i=1}^m  \|u_i \|^2 + \frac{\ve b}{2} \sum_{i=1}^m \|w_i \|^2 \\
	+ &\, 2d\,C_1p \,  \sum_{i=1}^m \sum_{j=1}^m \int_{\G_{ij}} ( u_i - u_j)^2\, dx \leq  2C_1 m \|\vp \|^2 + 2C_2 m |\gw |,  \quad  t \in I_{max},
	\end{split}
\eeq
where
$$
	C_2 = C_1 +  \frac{C_1 \gl}{2} + \frac{C_1 J^2}{2\gl} + \frac{\ve a^2}{b} + \frac{4\ve \gs^2}{\gl^2 b^3}.
$$
Consequently, \eqref{u2w} gives rise to the Gronwall-type differential inequality
	\beq \bl{GZ}
	\begin{split}
	 &\frac{d}{dt} \left[C_1 \sum_{i = 1}^m \|u_i\|^2 + \sum_{i=1}^m \| w_i \|^2 \right] + r \left[ C_1 \sum_{i = 1}^m \|u_i\|^2 + \sum_{i=1}^m \| w_i \|^2 \right]  \\
	\leq &\, \frac{d}{dt} \left[C_1 \sum_{i = 1}^m \|u_i\|^2  + \sum_{i=1}^m \| w_i \|^2 \right] + \frac{2\ve}{b} \sum_{i=1}^m \|u_i \|^2 + \frac{\ve b}{2} \sum_{i=1}^m \|w_i \|^2\\[3pt]
	\leq &\, 2C_1 m \|\vp \|^2 + 2C_2 m |\gw |,  \qquad  t \in I_{max} = [0, T_{max}),
	\end{split}
	\eeq
where 
$$
	r = \min \, \left\{ \frac{2\ve}{C_1 b}, \, \frac{\ve b}{2}\right\} = \min \, \left\{ \frac{4\gs^2}{\gl b^2}, \, \frac{\ve b}{2}\right\}.
$$
We can solve the differential inequality \eqref{GZ} to obtain the following bounding estimate of all the weak solutions on the maximal existence time interval $I_{max}$,
\beq \label{dse}
	\begin{split}
		&\sum_{i=1}^m \|g_i (t, g_i^0)\|^2 = \sum_{i=1}^m \left(\|u_i (t, u_i^0)\|^2 + \|w_i (t, w_i^0)\|^2 \right) \\
		\leq &\, \frac{1}{\min \{C_1, 1\}} e^{- r t} \sum_{i=1}^m \left(C_1\|u_i^0 \|^2 + \|w_i^0\|^2 \right) + \frac{2m}{r \min \{C_1, 1\}}\left( C_1 \|\vp\|^2 + C_2|\gw |\right)  \\
		\leq &\, \frac{\max \{C_1, 1\}}{\min \{C_1, 1\}}e^{- r t} \sum_{i=1}^m \|g_i^0\|^2 +  \frac{2m}{r \min \{C_1, 1\}} \left( C_1 \|\vp\|^2 + C_2|\gw |\right), \quad t \in [0, \infty).
	\end{split}
\eeq
Here it is shown that $I_{max} = [0, \infty)$ for every weak solution $g(t, g^0)$ because it will never blow up at any finite time. Therefore, for any initial data $g^0 = (g_1^0, \cdots, g_m^0) \in \mH$, the weak solution of the initial value problem \eqref{pb} of this neural network \eqref{cHR}-\eqref{inc} exists in the space $\mH$ for $t \in [0, \infty)$.
\end{proof}

The global existence and uniqueness of the weak solutions and their continuous dependence on the initial data enable us to define the solution semiflow $\{S(t): \mH \to \mH\}_{t \geq 0}$ of this boundary coupled FitzHugh-Nagumo neural network \eqref{cHR}-\eqref{inc} to be
\beq \bl{HRS}
	S(t): (g_1^0, \cdots, g_m^0) \longmapsto (g_1(t, g_1^0), \cdots, g_m (t, g_m^0)), \quad  t \geq 0.
\eeq
We shall call this semiflow $\{S(t)\}_{t \geq 0}$ on the space $\mH$ the \emph{boundary coupling FitzHugh-Nagumo semiflow}. 

The next result demonstrates the dissipative dynamics of this semiflow.
\begin{theorem} \label{Hab}
	There exists an absorbing set for the boundary coupling FitzHugh-Nagumo semiflow $\{S(t)\}_{t \geq 0}$ in the space $\mH$, which is the bounded ball 
\beq \label{abs}
	B^* = \{ h \in \mH: \| h \|^2 \leq Q\}
\eeq 
	where the constant 
$$
	Q = 1 +  \frac{2m}{r \min \{C_1, 1\}} \left( C_1 \|\vp\|^2 + C_2|\gw |\right).
$$ 
\end{theorem}

\begin{proof}
This is the consequence of the uniform estimate \eqref{dse}, which implies that
	\beq \label{lsp}
	\limsup_{t \to \infty} \, \sum_{i=1}^m \|g_i(t, g_i^0)\|^2 < Q 
	\eeq
for all weak solutions of \eqref{pb} with any initial data $(g_1^0, \cdots, g_N^0)$ in $\mH$. Moreover, for any given bounded set $B = \{h \in \mH: \|h \|^2 \leq \rho \}$ in $\mH$, there exists a finite time 

	\beq \label{T0B}
	T_B = \frac{1}{r} \log^+ \left(\rho \, \frac{\max \{C_1, 1\}}{\min \{C_1, 1\}}\right)
	\eeq
	such that all the solution trajectories started at the initial time $t = 0$ from the set $B$ will permanently enter the bounded ball $B^*$ shown in \eqref{abs} for $t \geq T_B$.  According to Definition \ref{Dabsb}, the bounded ball $B^*$ is an absorbing set in $\mH$ for the semiflow $\{S(t)\}_{t \geq 0}$ and this boundary coupling FitzHugh-Nagumo semiflow is dissipative.
\end{proof}

\section{\textbf{Coupling Dynamics of the Neural Network}}

In this section, we first prove a theorem on the ultimately uniform bound of all the weak solutions in the product space $L^4 (\gw) \times L^2(\gw)$. Through this bridge, we are able to show that all the weak solutions are ultimately uniform bounded in the regular space $E = H^1 (\gw) \times L^2 (\gw)$. Then by the trace theorem of Sobolev spaces, it results in an estimate of the ultimately uniform bound of the boundary coupling integrals of the $u_i$-components in the trace space $L^2 (\G)$ for this FitzHugh-Nagumo neural network.

\begin{theorem} \bl{Tm4}
	There exists a constant $L > 0$ shown in \eqref{L} such that for any initial data $g^0 = (g_1^0, \cdots , g_m^0) \in \mathbb{H}$,  the weak solution $g(t, g^0) = (g_1 (t, g_1^0), \cdots, g_m (t, g_m^0))$ of the initial value problem \eqref{pb} for the boundary coupled FitzHugh-Nagumo neural network \eqref{cHR}-\eqref{inc} satisfies the absorbing property in the space $L^4(\gw) \times L^2 (\gw)$,
\beq \bl{Lbd}
	\limsup_{t \to \infty}\, \sum_{i=1}^m \left(\|u_i(t)\|^4_{L^4} + \|w_i(t)\|^2 \right) \leq L. 
\eeq
\end{theorem}

\begin{proof}
Take the $L^2$ inner-product of the $u_i$-equation in \eqref{cHR} with $u_i^3, 1 \leq i \leq m$, and sum them up. By the boundary condition \eqref{nbc} and Assumption \eqref{Asp}, we get
\beq \label{u2}
	\begin{split}
	&\frac{1}{4}\, \frac{d}{dt} \sum_{i=1}^m \|u_i(t)\|^{4}_{L^{4}} + 3d\, \sum_{i=1}^m \|u_i \nb u_i \|^2_{L^2} \\
	+ &\, dp \, \sum_{i=1}^m \sum_{j=1}^m \int_{\G_{ij}} (u_i - u_j)^2 (u_i^2 + u_i u_j + u_j^2)\, dx \\
	= &\, \sum_{i=1}^m \int_\gw (f(u_i, x) u_i^3 - \gs u_i^3 w_i + J u_i^3)\, dx \\
	\leq &\, \sum_{i=1}^m \int_\gw (- \gl u_i^6 + u_i^2 \vp (x) - \gs u_i^3 w_i + Ju_i^3)\, dx, \quad t > 0.
	\end{split}
\eeq
By Cauchy inequality, it is seen that
\begin{equation*}  
	u_i^2 \vp (x) - \gs u_i^3 w_i + Ju_i^3 \leq \frac{1}{6} \gl u_i^4 + \frac{1}{3} \gl u_i^6 + \frac{6}{\gl} \left(\vp^2 (x) + \gs^2 |w_i(t,x)|^2 + J^2 \right).
\end{equation*}
Since $u_i^2 + u_i u_j + u_j^2 \geq 0$ always holds for all $1 \leq i, j \leq m$, from \eqref{u2} and the above inequality it follows that 
\beq \label{u3}
	\begin{split}
	&\frac{1}{4} \, \frac{d}{dt} \sum_{i=1}^m \|u_i(t)\|^{4}_{L^{4}} + 3d\, \sum_{i=1}^m \|u_i \nb u_i \|^2_{L^2} \\
	\leq &\, \sum_{i=1}^m  \left(\gl \int_\gw \left(\frac{1}{6} u_i^4 - \frac{2}{3} u_i^6\right) dx + \frac{6 \gs^2}{\gl} \|w_i (t)\|^2 \right)+ \frac{6 m}{\gl} \left(\|\vp \|^2 + J^2 |\gw |\right) \\
	\leq &\, \sum_{i=1}^m  \left( - \frac{\gl}{2} \int_\gw u_i^6 \, dx + \frac{6 \gs^2}{\gl} \|w_i (t)\|^2 \right)+ \frac{6 m}{\gl} \left(1 + \|\vp \|^2 + J^2 |\gw |\right), \quad t > 0, 
	\end{split}
\eeq
where $u_i^4 \leq \frac{1}{3} + \frac{2}{3} u_i^6 \leq 1 + u_i^6$ is used in the last step. 

Then take the $L^2$ inner-product of the $w_i$-equation in \eqref{cHR} with $C_3 w_i, 1 \leq i \leq m$, and sum them up. The positive constant $C_3$ is to be chosen. Then we have
\beq \bl{w2iq}
	\begin{split}
	&\frac{1}{2} \, \frac{d}{dt} \sum_{i=1}^m C_3 \| w_i \|^2 = \sum_{i=1}^m \int_\gw C_3 (\ve u_i w_i + \ve a w_i - \ve bw_i^2) \, dx  \\[2pt]
	\leq &\, \sum_{i=1}^m \int_\gw C_3 \left[\left(\frac{\ve}{b}\, u_i^2 + \frac{1}{4} \ve b \,w_i^2\right)  + \left(\frac{\ve a^2}{b} + \frac{1}{4} \ve b \,w_i^2\right) - \ve b \,w_i^2\right] dx \\[3pt]
	= &\, \sum_{i=1}^m \int_\gw \frac{C_3 \ve}{b}\, u_i^2 (t, x)\, dx - \frac{C_3}{2} \ve b \, \sum_{i=1}^m \|w_i\|^2 + \frac{C_3 m\, \ve a^2}{b}\, |\gw |, \quad t > 0.
	\end{split}
\eeq
Now add up \eqref{u3} and \eqref{w2iq} to obtain
\beq \bl{uwe}
	\begin{split}
	&\frac{1}{4}\, \frac{d}{dt} \sum_{i=1}^m \left(\|u_i(t)\|^{4}_{L^{4}} + 2C_3 \|w_i(t)\|^2 \right)+ 3d\, \sum_{i=1}^m \|u_i \nb u_i \|^2_{L^2} \\[3pt]
	\leq &\, \sum_{i=1}^m \left[\int_\gw \left(\frac{C_3 \ve}{b}\, u_i^2 (t, x) - \frac{\gl}{2}\, u_i^6 (t,x) \right) dx + \left(\frac{6 \gs^2}{\gl} - \frac{C_3 \ve b}{2}\right)\|w_i (t)\|^2 \right] \\[3pt]
	&\, + m \left(\frac{6}{\gl} (1 + \|\vp \|^2 + J^2 |\gw |) + \frac{C_3 \ve a^2}{b} |\gw |\right), \quad \text{for} \;\,  t > 0. \\
	\end{split}
\eeq	
Due to $u_i^6 - 2u_i^4 + u_i^2 = u_i^2 (u_i^2 - 1)^2 \geq 0$ so that $u_i^6 \geq 2u_i^4 - u_i^2$, we see that
\beq \bl{u64}
	\begin{split}
	&\frac{C_3 \ve}{b} u_i^2 (t, x) - \frac{\gl}{2} u_i^6 (t,x) = - \frac{\gl}{2} \left(u_i^6 - \frac{2C_3 \ve}{b \gl} u_i^2 \right) \leq - \frac{\gl}{2} \left(2u_i^4 - u_i^2 - \frac{2C_3 \ve}{b \gl} u_i^2 \right) \\
	= &\, - \frac{\gl}{2} u_i^4 - \frac{\gl}{2} \left[u_i^4 - \left(1 + \frac{2C_3 \ve}{b\gl} \right)u_i^2\right] \\
	= &\, - \frac{\gl}{2} u_i^4 - \frac{\gl}{2} \left[u_i^2 - \frac{1}{2}\left(1 + \frac{2C_3 \ve}{b\gl} \right)\right]^2 + \frac{\gl}{8}\left(1 + \frac{2C_3 \ve}{b\gl} \right)^2 \leq - \frac{\gl}{2} u_i^4 + \frac{\gl}{8}\left(1 + \frac{2C_3 \ve}{b\gl} \right)^2
	\end{split}
\eeq
Now we choose 
\beq \bl{C3w}
	C_3 = \frac{14 \gs^2}{\ve b\gl} \quad \text{so that} \quad \frac{6 \gs^2}{\gl} - \frac{C_3 \ve b}{2} = - \frac{\gs^2}{\gl}. 
\eeq
Substitute \eqref{u64} and \eqref{C3w} into \eqref{uwe} and delete the nonnegative term $3d \sum_i \|u_i \nb u_i\|^2$ from \eqref{uwe}. We end up with the differential inequality

\beq \bl{L4q}
	\begin{split}
	&\frac{1}{4}\,\frac{d}{dt} \sum_{i=1}^m \left(\|u_i(t)\|^{4}_{L^{4}} + 2C_3 \|w_i(t)\|^2 \right) + \sum_{i=1}^m \left(\frac{\gl}{2} \,\|u_i (t)\|^4_{L^4} + \frac{\gs^2}{\gl} \|w_i (t)\|^2\right) \\
	\leq &\, m \left[\frac{6}{\gl} (1 + \|\vp \|^2) + \left(\frac{6J^2}{\gl} + \frac{C_3 \ve a^2}{b} +  \frac{\gl}{8}\left(1 + \frac{2C_3 \ve}{b\gl} \right)^2 \right) |\gw| \right], \quad  \text{for} \; \; t > 0. 
	\end{split}
\eeq	
Set the constant
$$
	\delta = 4 \min \, \left\{\frac{\gl}{2}, \, \frac{\gs^2}{2C_3 \gl} \right\} = 2 \min \, \left\{\gl, \, \frac{\gs^2}{C_3 \gl} \right\}.
$$
Then \eqref{L4q} yields 
\beq \bl{Guw}
	\begin{split}
	&\frac{d}{dt} \sum_{i=1}^m \left(\|u_i(t)\|^{4}_{L^{4}} + 2C_3 \|w_i(t)\|^2 \right) + \delta \sum_{i=1}^m \left(\|u_i (t)\|^4_{L^4} + 2C_3 \|w_i (t)\|^2\right) \\
	\leq &\, m \left[\frac{24}{\gl} (1 + \|\vp \|^2) + \left[\frac{24 J^2}{\gl} + \frac{4C_3 \ve a^2}{b} +  \frac{\gl}{2} \left(1 + \frac{2C_3 \ve}{b\gl} \right)^2 \right] |\gw| \right], \;\,  \text{for} \; \; t > 0. 
	\end{split}
\eeq

By the parabolic regularity stated in Proposition \ref{pps} that any weak solution satisfies $g(\cdot ,  g^0) \in L^2([0, 1], \mathbb{E})$, there must be a time point $\tau_i \in (0, 1)$ such that $g_i(\tau_i, g_i^0) \in \mathbb{E}$ so that $u_i (\tau_i) \in H^1 (\gw) \subset L^4 (\gw)$. Then the second statement in Proposition \ref{pps} shows that any weak solution becomes a strong solution on the time interval $[1, \infty)$ so that 
$$
	(u_i , w_i) \in C([1, \infty); H^1 (\gw) \times L^2 (\gw)) \subset C([1, \infty); L^4 (\gw) \times L^2 (\gw)). 
$$

Finally apply the Gronwall inequality to \eqref{Guw}. Then we achieve the bounding estimate of all the solutions in the space $L^4 (\gw) \times L^2 (\gw)$:
\beq \bl{L4B}
	\begin{split}
	& \sum_{i=1}^m \left(\|u_i(t)\|^{4}_{L^{4}} + \|w_i(t)\|^2 \right) \leq \frac{1}{\min \{1, 2C_3\}} \sum_{i=1}^m \left(\|u_i(t)\|^{4}_{L^{4}} + 2C_3 \|w_i(t)\|^2 \right) \\
        \leq &\, \frac{e^{- \delta (t - \tau)}}{\min \{1, 2C_3\}} \sum_{i=1}^m \left(\|u_i(\tau_i)\|^{4}_{L^{4}} + 2C_3 \|w_i(\tau_i)\|^2 \right) + L, \quad \text{for} \;\, t \geq 1 > \tau_i,
	\end{split}
\eeq
where the positive constant $L$ is independent of any initial data,
\beq \bl{L}
	L = \frac{m}{\delta \min \{1, 2C_3\}} \left[\frac{24}{\gl} (1 + \|\vp \|^2) + \left[\frac{24J^2}{\gl} + \frac{4C_3 \ve a^2}{b} +  \frac{\gl}{2}\left(1 + \frac{2C_3 \ve}{b\gl} \right)^2 \right] |\gw| \right]
\eeq 
for any initial data $g^0 = (g_1^0, \cdots , g_m^0)$ in $\mathbb{H}$. This completes the proof of \eqref{Lbd} by taking the limit as $t \to \infty$ in \eqref{L4B}.
\end{proof}

The next result is to prove that the $u_i$-components of the weak solutions in the boundary coupling are ultimately and uniformly bounded in the regular interior space $H^1 (\gw)$ and to provide an estimate of the uniform bound. 

\begin{theorem} \bl{TmB}
	There exists a a constant $K > 0$ shown in \eqref{ET} such that for any initial data $g^0 = (g_1^0, \cdots , g_m^0) \in \mathbb{H}$, the $u_i$-components of the weak solution $g(t, g^0) = (g_1 (t, g_1^0), \cdots, g_m (t, g_m^0))$ of the initial value problem \eqref{pb} for the boundary coupled FitzHugh-Nagumo neural network \eqref{cHR}-\eqref{inc} satisfies
\beq \bl{H1}
	\limsup_{t \to \infty} \|g_i(t, g_i^0) - g_k (t, g_k^0)\|_E \leq K, \quad \text{for} \;\; 1 \leq i, k \leq m.
\eeq
\end{theorem}

\begin{proof}
	By the Assumption \eqref{Asp}, the Nemytskii operator $F: u(x) \mapsto f(u(x), x), x \in \gw,$ shown in \eqref{opf} has the Lipschitz property: For any $g = (u, w)$ and $\widetilde{g} = (\widetilde{u}, \widetilde{w})$, 
\begin{equation} \bl{Lip}
	\begin{split}
	&\|F(g) - F(\widetilde{g})\| \leq \|f(u(x), x) - f(\widetilde{u}(x), x)\| + \ve \|u - \widetilde{u}\| + \gs \|w - \widetilde{w}\|  \\
	= &\, \left\|\frac{\pdr f}{\pdr s} \left(\theta u(x) + (1 - \theta)\widetilde{u}(x), x\right) (u - \widetilde{u})\right\| + \ve \|u - \widetilde{u}\| + \gs \|w - \widetilde{w}\|  \\
	\leq &\,\|\beta (\theta u + (1- \theta) \widetilde{u})^2 + \xi (x)\| \|u - \widetilde{u}\| + \ve \|u - \widetilde{u}\| + \gs \|w - \widetilde{w}\|   \\[3pt]
	\leq &\,C_4 \left(\|u\|^2_{L^4} + \|\widetilde{u}\|^2_{L^4} + \|\xi\| \right) \|u - \widetilde{u}\| + \ve \|u - \widetilde{u}\| + \gs \|w - \widetilde{w}\| 
	\end{split}
\end{equation}
where $0 \leq \theta \leq 1$. The constant $\beta > 0$ and the function $\xi (x)$ are given in \eqref{Asp}. $C_4 (\beta) > 0$ is a constant. Let $B^{**}$ be the bounded ball in $[L^4 (\gw)\times L^2(\gw)]^m$ given by
\beq \bl{ab4}
	B^{**} = \left\{h \in [L^4 (\gw) \times L^2 (\gw)]^m: \sum_{i=1}^m (\|h_{iu}\|^4_{L^4} + \|h_{iw}\|^2) \leq \max \,\{L, \gs\} \right\}.
\eeq
If $g, \widetilde{g} \in B^{**}$, then \eqref{Lip} implies that
\beq \bl{Fg}
	\begin{split}
	\|F(g) - F(\widetilde{g})\| &\leq (C_4 (2\sqrt{L} + \|\xi\|) + \ve + \gs) \|g - \widetilde{g}\| \\[2pt]
	& \leq (C_4 (2\sqrt{L} + \|\xi\|) + \ve + \gs) \|g - \widetilde{g}\|_E.
	\end{split}
\eeq

Since the weak solutions $g_i(t, g_i^0), 1 \leq i \leq m$, are also mild solutions, we can write
\beq \bl{Msn}
	g_i(t, g_i^0) = e^{A_i(t - \tau)} g_i(\tau, g_i^0) + \int_\tau^t e^{A_i(t-s)} F(g_i(s, g_i^0))\, ds, \quad \text{for}\; t > \tau \geq 1.
\eeq
The analytic $C_0$-semigroup $\{e^{A_it}\}_{t \geq 0}$ is generated by the operator $A_i$ in \eqref{opA}. Since $A_i$ is a nonpositive, self-adjoint, closed linear operator with compact resolvent, without loss of generality, we can assume that all the eigenvalues of the operator $A_i$ are $\{- \mu_k\}$ and $\mu_k \geq \alpha_0 > 0$. Otherwise, we can simply replace $A_i$ by $A_i - \alpha_0 I$ and replace the nonlinear term $F(g_i)$ by $F(g_i) + \alpha_0\, g_i$ in \eqref{opf}, which still satisfies the Assumption \eqref{Asp}. 

By the analytic semigroup regularity \cite[Theorem 37.5]{SY} and the above spectrum property of the operators $A_i$, there is a constant $c > 0$ such that
$$
	\|e^{A_i t}\|_{\mathcal{L}(H, E)} \leq c\, e^{- \alpha_0\, t}\, t^{-t/2}, \quad  t > 0. 
$$
Thus for any solution $g_i(t, g_i^0), 1 \leq i \leq m,$ of \eqref{pb} with the initial data $g_i^0 \in H$, by Theorem \ref{Hab} and Theorem \ref{Tm4}, there is a finite time $T = T(g_1^0, \cdots, g_m^0) \geq 1$ such that the solution $(g_1(t, g_1^0), \cdots, g_m(t, g_m^0)) \in B^* \cap B^{**}$ for all $t \geq T$, where $B^*$ and $B^{**}$ are defined by \eqref{abs} and \eqref{ab4} respectively. Using the Lipschitz property \eqref{Fg}, it holds that for any $1 \leq i, k \leq m$ and for $t > T$,
\beq \bl{sgp}
	\begin{split}
	&\|g_i(t, g_i^0) - g_k(t, g_k^0)\|_E =  \|e^{A_i(t - T)}\|_{\mathcal{L}(H, E)} \|g_i(T, g_i^0) - g_k(T, g_k^0)\| \\[2pt]
	&+ \int_T^t \|e^{A_i(t-s)}\|_{\mathcal{L}(H, E)} \|F(g_i(s, g_i^0)) - F(g_k(s, g_k^0)\|\, ds \\
	= &\, \frac{c\, e^{-\alpha_0 (t - T)}}{\sqrt{t - T}} \,\|g_i(T, g_i^0) - g_k(T, g_k^0)\| + \int_T^t \frac{c \,e^{- \alpha_0 (t-s)}}{\sqrt{t - s}}\|F(g_i(s, g_i^0)) - F(g_k(s, g_k^0)\|\, ds \\
	\leq &\,\frac{2c\, Q e^{-\alpha_0 (t - T)}}{\sqrt{t - T}} + \int_T^t \frac{c \,e^{- \alpha_0 (t-s)}}{\sqrt{t - s}}\left[C_4 (2\sqrt{L} + \|\xi \|) + \ve + \gs \right] \|g_i(s, g_i^0) - g_k(s, g_k^0)\|\, ds.
	\end{split}
\eeq
Denote by $q(t) = e^{\alpha_0 \, t} \|g_i(t, g_i^0)- g_k(t, g_k^0)\|_E$ and set the constant
$$
	M = C_4 (2\sqrt{L} + \|\xi \|) + \ve + \gs.
$$ 
Then \eqref{sgp} infers that $q(t)$ satisfies the following integral inequality
\beq \bl{qt}
	q(t) \leq \frac{2c\, e^{\alpha_0 T} Q}{\sqrt{t -T}} + \int_T^t \frac{c M}{\sqrt{t - s}} \,q(s)\, ds, \quad t > T.
\eeq
Apply the integral Gronwall inequality \cite[page 9]{HA} to \eqref{qt}. It yields
\beq \bl{GI}
	\begin{split}
	q(t) &\leq \frac{2c\, e^{\alpha_0 T} Q}{\sqrt{t -T}} + \int_T^t \frac{c M}{\sqrt{t - s}} \,\frac{2c\, e^{\alpha_0 T} Q}{\sqrt{s-T}}\, \exp \left[\int_s^t \frac{c M}{\sqrt{t - \rho}}\, d\rho \right] ds \\[3pt]
	&= \frac{2c\, e^{\alpha_0 T} Q}{\sqrt{t -T}} + \int_T^t \frac{2c^2\,e^{\alpha_0 T}MQ}{\sqrt{(t - s)(s - T)}} \, \exp \left(2c M \sqrt{t - s} \,\right) ds \\[3pt]
	&= \frac{2c\, e^{\alpha_0 T} Q}{\sqrt{t -T}} + \int_0^{t-T} \frac{2c^2\,e^{\alpha_0 T}MQ}{\sqrt{s (t - T - s)}}\, \exp \left(2c M \sqrt{s} \,\right) ds, \quad \text{for} \;\; t > T.
	\end{split}
\eeq
Therefore, for all $1 \leq i, k \leq m$,

\beq \bl{Ebd}
	\begin{split}
	&\|g_i(t, g_i^0)- g_k(t, g_k^0)\|_E = e^{- \alpha_0 \, t} q(t) \\[8pt]
	\leq &\, \frac{2c\, e^{- \alpha_0 (t - T)} Q}{\sqrt{t -T}} + \int_0^{t-T} \frac{2c^2\,e^{- \alpha_0 (t - T)}MQ}{\sqrt{s (t - T - s)}}\, e^{2c M \sqrt{s}} ds \\
	\leq &\, \frac{2c\, e^{- \alpha_0 (t - T)} Q}{\sqrt{t -T}} + 2c^2 MQ\, e^{\left(- \alpha_0 (t - T) + 2c M \sqrt{t - T} \right)} \int_0^{t-T} \frac{1}{\sqrt{s (t - T - s)}}\, ds \\
	\leq &\, \frac{2c\, Q}{\sqrt{t -T}} + 2c^2 MQ\, \exp \left[- \alpha_0 \left(\sqrt{t-T} - \frac{c M}{\alpha_0} \right)^2 + \frac{c^2 M^2}{\alpha_0}\right] \widetilde{B} \left(\frac{1}{2},\, \frac{1}{2}\right) \\
	\leq &\, \frac{2c\, Q}{\sqrt{t -T}} + 2c^2 MQ\, \exp \left[- \alpha_0 \left(\sqrt{t-T} - \frac{c M}{\alpha_0} \right)^2 + \frac{c^2 M^2}{\alpha_0}\right] \widetilde{\G} \left(\frac{1}{2}\right)^2 \\
	\leq &\, \frac{2c\, Q}{\sqrt{t -T}} + 2c^2 MQ\, \pi \, \exp \left[ \frac{c^2 M^2}{\alpha_0}\right], \quad \text{for} \;\; t > T.
	\end{split}
\eeq
In \eqref{Ebd}, $T \geq 1$ is specified earlier. Here $\widetilde{B}(\cdot, \cdot)$ is the Beta function and $\widetilde{\G} (\cdot)$ is the Gamma function with the fact that $\widetilde{\G}(1/2) = \sqrt{\pi}$. Thus \eqref{H1} is proved, since
\beq \bl{ET}
	\|g_i(t, g_i^0)- g_k(t, g_k^0)\|_E \leq K = 2c\, Q + 2c^2 MQ\, \pi \, \exp \left[ \frac{c^2 M^2}{\alpha_0}\right], \;\; \text{for} \;\; t \geq 2 T.
\eeq 
The proof is completed.
\end{proof}

\begin{corollary} \bl{Corl}
	There exists a constant $\Pi > 0$ such that for any initial data $g^0 = (g_1^0, \cdots , g_m^0) \in \mathbb{H}$, the boundary trace of the coupling $u_i$-components of the weak solution $g(t, g^0) = (g_1 (t, g_1^0), \cdots, g_m (t, g_m^0))$ of the initial value problem \eqref{pb} satisfies
\beq \bl{bdbd}
	\limsup_{t \to \infty}\, \int_\G |u_i (t,x) - u_k(t, x)|^2\, dx \leq \Pi, \quad 1 \leq i, k \leq m.
\eeq
\end{corollary}

\begin{proof}
By the trace theorem of Sobolev spaces, since the $\gw$ is a locally Lipschitz domain, the trace operator $\mathcal{T} (g) = g|_{\G}: H^s (\gw) \to H^{s - \frac{1}{2}} (\G)$ is a bounded linear operator for any $\frac{1}{2} < s \leq 1$. Thus there is a constant $C^* > 0$ such that $\|\mathcal{T} (g) \|^2_{L^2 (\G)} \leq C^* \|g\|^2_{H^1 (\gw)}$, for any $g \in H^1 (\gw)$.  From \eqref{ET} we reach the proof of \eqref{bdbd}, 
\beq \bl{Pi}
	\begin{split}
	&\limsup_{t \to \infty}\, \int_\G |u_i (t,x) - u_k (t, x)|^2\, dx = \limsup_{t \to \infty} \| \mathcal{T} (u_i (t) - u_k (t)) \|^2_{L^2 (\G)}  \\
	\leq C^* &\, \limsup_{t \to \infty} \|u_i(t) - u_k (t)\|^2_{H^1 (\gw)} \leq C^* \limsup_{t \to \infty} \|g_i (t, g_i^0) - g_k(t, g_k^0)\|^2_E \leq \Pi,  
	\end{split}
\eeq
for $1 \leq i, k \leq m$, where $\Pi = C^* K^2$. It completes the proof.
\end{proof}

\section{\textbf{Synchronization of the FitzHugh-Nagumo Neural Network}} 

For mathematical models of biological neural networks, we define the asynchronous degree for study of the synchronization dynamics.
\begin{definition} \bl{DaD}
	For the dynamical system generated by a model differential equation such as \eqref{pb} of a dynamic neural network with whatever type of coupling, define the \emph{asynchronous degree} in a state Banach space $\ms{X}$ to be
	$$
	deg_s (\ms{X})= \sum_{i} \sum_{j} \, \sup_{g_i^0, \, g_j^0 \in \ms{X}} \, \left\{\limsup_{t \to \infty} \, \|g_i (t) - g_j(t)\|_{\ms{X}}\right\},
	$$ 
where $g_i (t)$ and $g_j (t)$ are any two solutions of the model differential equation for two neurons in the network with the initial states $g_i^0$ and $g_j^0$, respectively. The neural network is said to be synchronized in the space $\ms{X}$, if $deg_s (\ms{X}) = 0$.
\end{definition}

We now use the leverage of dissipative dynamics and the boundary integral estimates in the previous two sections to prove the main result on the asymptotic synchronization of the boundary coupled FitzHugh-Nagumo neural networks described by \eqref{cHR}-\eqref{inc} in the space $H$. This result provides a quantitative threshold for the coupling strength and the stimulation signals to reach the synchronization.

Set $U_{ij} (t,x) = u_i(t,x) - u_j (t,x)$ and $W_{ij} (t,x) = w_i(t,x) - w_i(t,x)$, for $i, j = 1, \cdots, m$. Given any initial states $g_1^0, \cdots, g_m^0$ in the space $H$, the difference between any two solutions of the model equation \eqref{pb} associated with the coupled neurons $\mathcal{N}_i$ and $\mathcal{N}_j$ in the network is what we consider: 
	$$
	g_i (t, g_i^0) - g_j (t, g_j^0) = \text{col}\, (U_{ij}(t, \cdot ), W_{ij}(t, \cdot )), \quad t \geq 0.
	$$
	
	By subtraction of the two equations of the $j$-th neuron from the corresponding equations of the $i$-th neuron in \eqref{cHR}, we obtain the following \emph{differencing} FitzHugh-Nagumo equations. For $i, j = 1, \cdots, m$,
\beq \bl{dHR}
	\begin{split}
		\frac{\pdr U_{ij}}{\pdr t} & = d \gd U_{ij} + f(u_i, x) - f(u_j, x) - \gs W_{ij},  \\
		\frac{\pdr W_{ij}}{\pdr t} & = \ve (U_{ij} - b W_{ij}).
	\end{split}
\eeq
Here is the main result of this work on the synchronization of the FitzHugh-Nagumo neural networks with the presented boundary coupling.

\begin{theorem} \bl{ThM}
	If the following threshold condition for stimulation signal strength of the boundary coupled FitzHugh-Nagumo neural network \eqref{cHR}-\eqref{inc} is satisfied,  
\beq \bl{SC}
	p\, \liminf_{t \to \infty}\, \sum_{1 \leq i < j \leq m} \int_\G U_{ij}^2(t, x)\, dx > R, 
\eeq
for any given initial conditions $(g_1^0, \cdots, g_m^0) \in \mathbb{H}$, where the constant
\beq \bl{R}
	R = m(m-1) \left[\left(\eta_2\, d |\gw | + \ga + 3|\ve - \gs |\right) \left[1 + \frac{2m}{r \min \{C_1, 1\}} (C_1 \|\vp\|^2 + C_2|\gw |) \right] \right],
\eeq
in which $\eta_2 > 0$ is given in the generalized Poincar\'{e} inequality \eqref{Pcr},
$$
	C_1 = \frac{\ve b \gl}{2\gs^2}    \quad    \text{and}   \quad     C_2 = C_1 +  \frac{C_1 \gl}{2} + \frac{C_1 J^2}{2\gl} + \frac{\ve a^2}{b} + \frac{4\ve \gs^2}{\gl^2 b^3},
$$
then this boundary coupled FitxHugh-Nagumo neural network is synchronized in the space $H$ at a uniform exponential rate.
\end{theorem}

\begin{proof}
	We go through the following three steps to prove this result.
	
	Step 1. Take the $L^2$ inner-products of the first equation in \eqref{dHR} with $U_{ij} (t)$ and the second equation in \eqref{dHR} with $W_{ij} (t)$. Then sum them up and use the Assumption  \eqref{Asp} to get
\beq \bl{eG}
	\begin{split}
		&\frac{1}{2} \frac{d}{dt} (\|U_{ij} (t)\|^2 + \|W_{ij} (t)\|^2) + d \|\nb U_{ij} (t)\|^2 + \ve b\,\|W_{ij} (t)\|^2 \\[3pt]
		= &\,\int_{\G}  \frac{\pdr U_{ij}}{\pdr \nu}\, U_{ij} \, dx +  \int_\gw (f(u_i, x) - f(u_j, x)) U_{ij}\, dx + \int_\gw (\ve - \gs)U_{ij} W_{ij} \, dx \\[3pt]
		\leq &\,\int_{\G} \frac{\pdr U_{ij}}{\pdr \nu}\, U_{ij} \, dx + \int_\gw  \frac{\pdr f}{\pdr s} \left(\xi u_i + (1-\xi) u_j, x \right) U^2_{ij}\, dx + \int_\gw (\ve - \gs)U_{ij} W_{ij}\, dx \\[3pt]
		\leq &\, \,\int_{\G} \frac{\pdr U_{ij}}{\pdr \nu}\, U_{ij} \, dx + \ga \|U_{ij}\|^2 + |\ve - \gs |( \|U_{ij}\|^2 + \| W_{ij}\|^2)
		\end{split}
\eeq
The boundary coupling condition \eqref{nbc} determines that
\beq \bl{bdt}
	\int_{\G} \frac{\pdr U_{ij}}{\pdr \nu}\, U_{ij} \, dx = -\, p \left[\sum_{k = 1}^m  \int_{\G_{ik}}\, (u_i - u_k) U_{ij} \, dx - \sum_{k=1}^m \int_{\G_{jk} }\, (u_j - u_k)] U_{ij} \, dx\right] = - p\, G_{ij}
\eeq	
where 
\beq \bl{Gij}
	G_{ij} = \sum_{k = 1}^m \int_{\G_{ik}} (u_i - u_k) (u_i - u_j) \, dx - \sum_{k = 1}^m \int_{\G_{jk}}  (u_j - u_k) (u_i - u_j)] \,dx.
\eeq
Substitute \eqref{bdt} into \eqref{eG}. Then we have
\begin{equation} \bl{meq}
	\begin{split}
	&\frac{1}{2} \frac{d}{dt} (\|U_{ij} (t)\|^2 + \|W_{ij} (t)\|^2) + d \|\nb U_{ij} (t)\|^2 + \ve b \|W_{ij} (t)\|^2 + p\, G_{ij} \\[3pt]
	\leq &\, \left(\ga + |\ve - \gs|\right) \|U_{ij} (t)\|^2 + |\ve - \gs |\|W_{ij} (t)\|^2 , \quad t > 0, \quad 1 \leq i, j \leq m.
	\end{split}
\end{equation}

Step 2. To treat the gradient term on the left-hand side of \eqref{meq}, we use the following generalized Poincar\'{e} inequality \cite{S}: There exist positive constants $\eta_1$ and $\eta_2$ depending only on the spatial domain $\gw$ and its dimension such that 
\beq \bl{Pcr}
	\eta_1 \|U_{ij} (t)\|^2 \leq \|\nb U_{ij} (t)\|^2 + \eta_2 \left[\int_\gw U_{ij} (t, x)\, dx\right]^2, \quad 1 \leq i, j \leq m.
\eeq
Moreover, \eqref{lsp} in Theorem \ref{Hab} confirms that 
\beq \bl{Lsup}
	\limsup_{t \to \infty}\, \sum_{i=1}^m \|g_i(t, g_i^0)\|^2 < Q = 1 + \frac{2m}{r \min \{C_1, 1\}} \left( C_1 \|\vp\|^2 + C_2|\gw |\right).
\eeq	
For all $1 \leq i, j \leq m$,
$$
	 \|U_{ij} (t)\|^2 \leq 2 \sum_{i=1}^m \|g_i(t, g_i^0)\|^2,  \quad \|W_{ij} (t)\|^2  \leq 2 \sum_{i=1}^m \|g_i(t, g_i^0)\|^2.
$$
Thus for any given bounded set $B \subset H$ and any initial data $g_i^0, g_j^0 \in B, 1 \leq i, j \leq m$, there is a finite time $T_B \geq 1$ depending on $B$ only such that 
\beq \bl{gbd}
	 \|U_{ij} (t)\|^2 \leq 2Q \quad \text{and} \quad  \|W_{ij} (t)\|^2 \leq 2Q, \quad \text{for} \;\, t > T_B.
\eeq
Therefore, \eqref{meq} combined with \eqref{Pcr} and \eqref{gbd} shows that, for any given bounded set $B \subset H$ and any initial data $g_i^0, g_j^0 \in B, 1 \leq i, j \leq m$, 
\beq \bl{Keq}
	\begin{split}
	& \frac{d}{dt} (\|U_{ij} (t)\|^2 + \|W_{ij} (t)\|^2) + 2\,\eta_1 d\, \|U_{ij} (t)\|^2 + 2 \ve b \|W_{ij} (t)\|^2 + 2p \,G_{ij} \\
	\leq &\, 2\eta_2 d \left[\int_\gw U_{ij} (t, x)\, dx\right]^2 + 2 \left(\ga + |\ve - \gs|\right) \|U_{ij} (t)\|^2 + 2 |\ve - \gs |\|W_{ij} (t)\|^2  \\[3pt]
	\leq &\, 2\eta_2 d\, \|U_{ij} (t)\|^2 |\gw | + 2 \left(\ga + |\ve - \gs|\right) \|U_{ij} (t)\|^2 + 2 |\ve - \gs |\|W_{ij} (t)\|^2 \\[10pt]
	= &\, 2\|U_{ij} (t)\|^2 \left(\eta_2\, d |\gw | + \ga + |\ve - \gs | \right)+ 2|\ve - \gs |\|W_{ij} (t)\|^2  \\[10pt]
	\leq &\, 4Q \left( \eta_2\, d |\gw | + \ga + 3|\ve - \gs |\right) \\[2pt]
	= &\, 4 \left(\eta_2\, d |\gw | + \ga + 3|\ve - \gs |\right) \left[1 + \frac{2m}{r \min \{C_1, 1\}} \left( C_1 \|\vp\|^2 + C_2|\gw |\right) \right], \; \; t > T_B.
	\end{split}
\eeq

	Step 3. To treat the ensemble coupling terms $2p\, G_{ij}$ in \eqref{Keq} for all $1 \leq i, j \leq m$, we define the characteristic functions $\psi_{ik} (x)$ on the boundary piece $\G_{ik}, 1 \leq i, k \leq m$, to be
$$
	\psi_{ik} (x) = \begin{cases}
	                      1, \qquad \text{if $x \in \G_{ik}$} , \\[2pt]
	                      0, \qquad \text{if $x \in \G \backslash \G_{ik}$} .
	                      \end{cases} 
$$
From \eqref{Gij} we can deduce that
\beq \bl{SmG}
	\begin{split}
	&\sum_{i, \,j} G_{ij} = \sum_{i, \,j} \left[\sum_{k = 1}^m \int_{\G_{ik}} (u_i - u_k) (u_i - u_j) \, dx - \sum_{k = 1}^m \int_{\G_{jk}}  (u_j - u_k) (u_i - u_j)] \,dx\right] \\
	= &\, \sum_{i, \,j} \left[\sum_{k = 1}^m \int_\G \psi_{ik} (u_i - u_k) (u_i - u_j) \, dx - \sum_{k = 1}^m \int_\G \psi_{jk}  (u_j - u_k) (u_i - u_j)] \,dx\right] \\
	= &\, \sum_{i, j} \int_{\G} \left[\sum_{k=1}^m \psi_{ik}(u_i -u_k)\right] (u_i - u_j) dx  - \sum_{i, j} \int_{\G} \left[\sum_{k=1}^m \psi_{jk} (u_j -u_k)\right] (u_i - u_j) dx \\
	= &\, \sum_{i, j} \int_{\G} \left[u_i - \sum_{k=1}^m \psi_{ik} u_k \right] (u_i - u_j)\, dx - \sum_{i, j} \int_{\G} \left[u_j - \sum_{k=1}^m \psi_{jk} u_k \right] (u_i - u_j)\, dx \\[3pt]
	= &\, \sum_{i=1}^m \sum_{j=1}^m \int_{\G} \left(u_i -  u_j \right) (u_i - u_j)\, dx - \sum_{i=1}^m \sum_{j=1}^m \int_{\G} \left(\widetilde{u_i} -  \widetilde{u}_j \right) (u_i - u_j)\, dx \\[3pt]
	= &\, \sum_{i=1}^m \sum_{j=1}^m \int_{\G} (u_i - u_j)^2\, dx, 
	\end{split}
\eeq
where $\widetilde{u_i} = \sum_{k=1}^m \psi_{ik} \, u_k = u_1\mid_{\G_{i1}} + \cdots + u_m \mid_{\G_{im}}, 1 \leq i \leq m$, is the sum of all the $u_k$-components of the solutions $g_k(t, g_k^0)$ distributed on the $i$-th neuron's coupling decomposition of the boundary $\G = \bigcup_{k=1}^m \G_{ik}$. And we have
$$
	\sum_{i=1}^N \sum_{j=1}^N  \int_{\G} \left(\widetilde{u_i} -  \widetilde{u}_j \right) (u_i - u_j)\, dx = \sum_{i=1}^N \left( \sum_{j < i} + \sum_{j > i} \right) \int_{\G} \left(\widetilde{u_i} -  \widetilde{u}_j \right) (u_i - u_j)\, dx = 0.
$$
Substitute \eqref{SmG} into the differential inequality \eqref{Keq} for all $1\leq i, j \leq m$ and sum them up. Then we get
\begin{align*}
	&\frac{d}{dt} \sum_{i=1}^m \sum_{j=1}^m (\|U_{ij} (t)\|^2 + \|W_{ij} (t)\|^2) + \sum_{i=1}^m \sum_{j=1}^m 2\left(\eta_1 d\, \|U_{ij} (t)\|^2 + \ve b \|W_{ij} (t)\|^2\right) \\
	&+ 2p\, \sum_{i=1}^m \sum_{j=1}^m \int_{\G} (u_i - u_j)^2\, dx \\
	\leq &\,4m(m-1) \left[\left(\eta_2\, d |\gw | + \ga + 3|\ve - \gs |\right) \left[1 + \frac{2m}{r \min \{C_1, 1\}} (C_1 \|\vp\|^2 + C_2|\gw |) \right] \right].
\end{align*}
Note that $U_{ii} = 0, W_{ii} = 0$ and $\|U_{ij}\| = \|U_{ji}\|, \|W_{ij}\| = \|W_{ji}\|$ for all $1 \leq i, j \leq m$. The above differential inequality is exactly equivalent to 
\beq  \bl{Mnq}
	\begin{split}
	\frac{d}{dt} \sum_{1 \leq i < j \leq m} & (\|U_{ij} (t)\|^2 + \|W_{ij} (t)\|^2) + 2 \sum_{1 \leq i < j \leq m} (\eta_1 d\, \|U_{ij} (t)\|^2 + 2\ve b \|W_{ij} (t)\|^2) \\
	&+ 2p \sum_{1 \leq i < j \leq m} \int_{\G} (u_i - u_j)^2\, dx \leq 2R,  \quad \text{for} \; \; t > T_B,  \\
	\end{split}
\eeq
where $R > 0$ is the constant in \eqref{R} and independent of any initial data. 

	Under the threshold condition \eqref{SC} of this theorem, for any given initial data $(g_1^0, \cdots, g_m^0) \in B$, there exists a sufficiently large $\tau = \tau (g_1^0, \cdots, g_m^0) > 1$ such that the ensemble stimulation signal strength of this boundary coupled neural network satisfies the threshold crossing inequality
	\beq \bl{pl}
	p\,  \sum_{1 \leq i < j \leq m} \int_\G U_{ij}^2(t, x)\, dx > R,  \quad \text{for} \;\; t > \tau (g_1^0, \cdots, g_m^0), \; 1 \leq i, j \leq m.
	\eeq
Then it follows from \eqref{Mnq} and \eqref{pl} that
\beq \bl{Gwq}
	\begin{split}
	& \frac{d}{dt} \sum_{1 \leq i < j \leq m} (\|U_{ij} (t)\|^2 + \|W_{ij} (t)\|^2) \\
	+ \, 2 & \min \{\eta_1 d,  \ve b\} \sum_{1 \leq i < j \leq m}(\|U_{ij} (t)\|^2 + \|W_{ij} (t)\|^2) < 0,
	\end{split}
\eeq
for $t >  \tau^* = \max \{\tau, T_B\}$. Finally, the Gronwall inequality applied to \eqref{Gwq} combined with \eqref{gbd} shows that
\begin{align*}
		& \sum_{1 \leq i < j \leq m} \left(\|U_{ij} (t)\|^2 + \|W_{ij} (t)\|^2\right)   \\
		\leq e^{- \mu (t - \tau^*)} \sum_{1 \leq i < j \leq m} &\, \left(\|U_{ij} (\tau^*)\|^2 + \|W_{ij} (\tau^*)\|^2\right) \leq 2m(m-1) e^{- \mu (t - \tau^*)}\,Q \to 0, 
\end{align*}
as $t \to \infty$, where $\mu = 2 \min \{\eta_1 d, \ve b\}$ is the exponential synchronization rate. Therefore, it is proved that
\beq \bl{gij}
	deg_s (\text{H})= \sum_{i = 1}^m \sum_{j=1}^m \,\sup_{g_i^0, g_j^0 \in H} \left\{ \limsup_{t \to \infty} \|g_i (t, g_i^0) - g_j (t, g_j^0)\|_H \right\} = 0.
\eeq
According to Definition \ref{DaD}, this boundary coupled FitzHugh-Nagumo neural network is synchronized in the space $H = L^2 (\gw, \mR^2)$ at a uniform rate. 
\end{proof}

This main theorem provides a sufficient condition for synchronization of the presented boundary coupled complex neural network. The biological interpretation of the threshold condition \eqref{SC} for synchronization is that the product of the \emph{boundary coupling strength} represented by the coupling coefficient $p$ and the \emph{ensemble boundary stimulation signals} represented by $\liminf_{t \to \infty}\, \int_\G \sum_{i< j} U_{ij}^2(t, x)\, dx$ for the network neurons exceeds the threshold constant $R$, which is explicitly expressed by the biological and mathematical parameters. It is certainly possible that the synchronization threshold can be reduced through further investigations.

The proof of \eqref{Keq} through \eqref{gij} for Theorem \ref{ThM} also shows that the presented complex neural networks can be partly synchronized if the condition \eqref{SC} is satisfied only for a subset of the neurons in the network. 

\bibliographystyle{amsplain}

\end{document}